\documentclass[12pt,reqno]{amsart}
\usepackage{amsmath,amsthm,amssymb,bbm,color,verbatim,tikz}
\input xy 
\xyoption{all}

\numberwithin{equation}{section}

\newcommand{\CARD}{{\rm CARD}}
\newcommand{\REG}{{\rm REG}}

\newcommand{\GCH}{{\rm GCH}}

\newcommand{\ZFC}{{\rm ZFC}}

\newcommand{\ORD}{\mathop{{\rm ORD}}}

\renewcommand{\P}{{\mathbb P}}
\newcommand{\Q}{{\mathbb Q}}
\newcommand{\R}{{\mathbb R}}

\newcommand{\Sacks}{\mathop{\rm Sacks}}
\newcommand{\Add}{\mathop{\rm Add}}

\newcommand{\forces}{\Vdash}
\newcommand{\forced}{\Vdash}
\renewcommand{\1}{\mathbbm{1}}

\newcommand{\restrict}{\upharpoonright}
\newcommand{\concat}{\mathbin{{}^\smallfrown}}

\newcommand{\supp}{\mathop{\rm supp}}
\newcommand{\Aut}{\mathop{\rm Aut}}
\newcommand{\dom}{\mathop{\rm dom}}
\newcommand{\ran}{\mathop{\rm ran}}

\newcommand{\cf}{\mathop{\rm cf}}

\renewcommand{\and}{\mathop{\&}}

\newcommand{\length}{\mathop{\rm length}}
\newcommand{\Split}{\mathop{\rm Split}}

\newtheorem{theorem}{Theorem}
\newtheorem{lemma}[theorem]{Lemma}
\newtheorem{corollary}[theorem]{Corollary}

\newtheorem{sublemma}{Sublemma}[theorem]

\newtheorem*{theorem31}{Theorem 3.1}

\theoremstyle{definition}
\newtheorem{question}{Question}

\newtheorem{fact}[theorem]{Fact}
\newtheorem{definition}[theorem]{Definition}

\thanks{This work was carried out while the author was a student under the advisement of Joel David Hamkins. The author wishes to thank Professor Hamkins for his guidance, as well as for many helpful conversations regarding the topics contained in this paper. The author also wishes to thank Arthur Apter for suggesting this course of research as well as for his helpful comments regarding Lemma \ref{lemmawoodinclosed} below.}
\date{\today}

\begin{document}

\title{Easton's Theorem in the presence of Woodin cardinals}

\author[Brent Cody]{Brent Cody}
\address[Brent Cody]{ 
The Fields Institute for Research in Mathematical Sciences, 
222 College Street, 
Toronto, Ontario M5S 2N2,
Canada} 
\email[B. ~Cody]{bcody@fields.utoronto.ca} 
\urladdr{http://www.fields.utoronto.ca/~bcody/}

\maketitle

\begin{abstract}
Under the assumption that $\delta$ is a Woodin cardinal and $\GCH$ holds, I show that if $F$ is any class function from the regular cardinals to the cardinals such that (1) $\kappa<\cf(F(\kappa))$, (2) $\kappa<\lambda$ implies $F(\kappa)\leq F(\lambda)$, and (3) $\delta$ is closed under $F$, then there is a cofinality-preserving forcing extension in which $2^\gamma= F(\gamma)$ for each regular cardinal $\gamma<\delta$, and in which $\delta$ remains Woodin. Unlike the analogous results for supercompact cardinals \cite{Menas:ConsistencyResultsConcerningSupercompactness} and strong cardinals \cite{FriedmanHonzik:EastonsTheoremAndLargeCardinals}, there is no requirement that the function $F$ be locally definable.

\end{abstract}

\section{Introduction}\label{sectionintroduction}

Easton \cite{Easton:PowersOfRegularCardinals} proved that the continuum function $\kappa\mapsto 2^\kappa$ on regular cardinals can be forced to behave in any way that is consistent with K\"onig's Theorem ($\kappa<\cf(2^\kappa)$) and monotonicity ($\kappa<\lambda$ implies $2^\kappa\leq 2^\lambda$). I will say that $F$ is an \emph{Easton function} if $F$ is a function from the class of regular cardinals to the class of cardinals satisfying (1) $\kappa<\cf(F(\kappa))$ and (2) $\kappa<\lambda$ implies $F(\kappa)\leq F(\lambda)$. In the presence of large cardinals, there are additional restrictions on the possible behaviors of the continuum function on regular cardinals. For example, Scott proved that if $\GCH$ fails at a measurable cardinal $\kappa$, then $\GCH$ fails on a normal measure one subset of $\kappa$. It seems natural to ask: 
\begin{question}\label{question}
Given a large cardinal $\kappa$, what Easton functions can be forced to equal the continuum function on the regular cardinals, while preserving the large cardinal property of $\kappa$? 
\end{question}

Menas \cite{Menas:ConsistencyResultsConcerningSupercompactness} showed that if $F$ is a ``locally definable'' Easton function (for a definition see \cite[Theorem 18]{Menas:ConsistencyResultsConcerningSupercompactness} or \cite[Definition 3.16]{FriedmanHonzik:EastonsTheoremAndLargeCardinals}), then there is a forcing extension $V[G]$ in which $2^\gamma=F(\gamma)$ for each regular cardinal $\gamma$ and each supercompact cardinal in $V$ remains supercompact in $V[G]$. In Menas' proof, the local definability of $F$ is needed to show that for an elementary embedding $j:V\to M$ witnessing the $\lambda$-supercompactness of $\kappa$, the functions $F$ and $j(F)$ agree to an extent allowing one to lift $j$ to $V[G]$. The developments in the literature addressing Question \ref{question} in the case where $\kappa$ is a measurable cardinal are more complicated. Woodin showed, using his method of modifying a generic filter, that if there is an elementary embedding $j:V\to M$ with critical point $\kappa$ such that $j(\kappa)>\kappa^{++}$ and $M^\kappa\subseteq M$ then there is a forcing extension in which $\kappa$ is measurable and $\GCH$ fails at $\kappa$ (see \cite[Theorem 25.1]{Cummings:Handbook} or \cite[Theorem 36.2]{Jech:Book}). In \cite{FriedmanThompson:PerfectTreesAndElementaryEmbeddings}, Friedman and Thompson introduced the tuning fork method and argued that it provides a more streamlined proof of Woodin's result. Friedman and Honzik \cite{FriedmanHonzik:EastonsTheoremAndLargeCardinals} made use of the uniformity of the tuning fork method and provided an answer to Question \ref{question} for measurable cardinals as well as for strong cardinals. Specifically, regarding strong cardinals, they proved that if $F$ is any locally definable Easton function and $\GCH$ holds, then there is a cofinality preserving forcing extension $V[G]$ in which $2^\gamma=F(\gamma)$ for each regular cardinal $\gamma$ and each strong cardinal in $V$ remains strong in $V[G]$.

In this paper I prove the following theorem, which provides a complete answer to Question \ref{question} for the case of Woodin cardinals (see Section \ref{sectionwoodin} below for a definition and general discussion of Woodin cardinals). 



\begin{theorem}\label{theoremwoodin}
Suppose $\GCH$ holds, $F:\REG\to\CARD$ is an Easton function, and $\delta$ is a Woodin cardinal closed under $F$. Then there is a cofinality-preserving forcing extension in which $\delta$ remains Woodin and $2^\gamma=F(\gamma)$ for each regular cardinal $\gamma$.
\end{theorem}


The proof of Theorem \ref{theoremwoodin} adapts the methods of \cite{FriedmanHonzik:EastonsTheoremAndLargeCardinals} and \cite{FriedmanThompson:PerfectTreesAndElementaryEmbeddings} to a new case. Notice that in Theorem \ref{theoremwoodin}, there is no requirement stating that $F$ must be \emph{locally definable} as in the results of \cite{Menas:ConsistencyResultsConcerningSupercompactness} and \cite{FriedmanHonzik:EastonsTheoremAndLargeCardinals}. It is the property $j(A)\cap\gamma = A\cap\gamma$ in one of the characterizations of Woodin cardinals (see Lemma \ref{lemmawoodin}) that allows the removal of this additional requirement on $F$. Since a straight forward argument shows that ${<}\delta$-closed forcing preserves the Woodinness of $\delta$ (see Lemma \ref{lemmawoodinclosed} below), the bulk of the work in proving Theorem \ref{theoremwoodin} will be to show that the continuum function can be forced to agree with $F$ below $\delta$, while preserving the Woodinness of $\delta$.

Let me remark here that as a corollary to the proof of Theorem \ref{theoremwoodin}, one has the following.

\begin{corollary}
Suppose $C$ is a class of Woodin cardinals and $F$ is an Easton function such that $\delta$ is closed under $F$ for each $\delta\in C$. Then there is a cofinality-preserving forcing extension in which $\delta$ remains Woodin and $2^\gamma=F(\gamma)$ for each regular cardinal $\gamma$.
\end{corollary}

\section{Preliminaries for the Proof of Theorem \ref{theoremwoodin}}\label{sectionpreliminaries}

\subsection{Lifting Embeddings}

In what follows, I will be concerned with arguing that the Woodinness of a cardinal is preserved through forcing. This property is witnessed by elementary embeddings $j:M\to N$ between models of set theory. To show that such a large cardinal property is preserved to a forcing extension, say $V[G]$, one typically lifts the embedding to $j^*:M[G]\to N[j(G)]$ and argues that the lifted embedding witnesses the large cardinal property in $V[G]$. In this section, I will present some standard lemmas that are useful for lifting embeddings. For proofs of Lemmas \ref{lemmaground} - \ref{lemmalambdadist}, one may consult \cite{Cummings:Handbook} or \cite{Cummings:AModelInWhichGCH}.


In what follows $N$ and $M$ are always assumed to be transitive models of $\ZFC$. The following two lemmas are useful for building generic objects.

\begin{lemma}\label{lemmaground}
Suppose that $M^\lambda\subseteq M$ in $V$ and there is in $V$ an $M$-generic filter $H\subseteq \Q$ for some forcing $\Q\in M$. Then $M[H]^\lambda\subseteq M[H]$ in $V$.
\end{lemma}

\begin{lemma}\label{lemmachain}
Suppose that $M\subseteq V$ is a model of $\ZFC$, $M^{<\lambda}\subseteq M$ in $V$ and $\P$ is $\lambda$-c.c. If $G\subseteq \P$ is $V$-generic, then $M[G]^{<\lambda}\subseteq M[G]$ in $V[G]$.
\end{lemma}

Suppose $j:M\to N$ is an embedding and $\P\in M$ a forcing notion. In order to lift $j$ to $M[G]$ where $G$ is $M$-generic for $\P$, one typically uses Lemmas \ref{lemmaground} and \ref{lemmachain} to build an $N$-generic filter $H$ for $j(\P)$ satisfying condition (1) in Lemma \ref{lemmaliftingcriterion} below.

\begin{lemma}\label{lemmaliftingcriterion}
Let $j:M\to N$ be an elementary embedding between transitive models of $\ZFC$. Let $\P\in M$ be a notion of forcing, let $G$ be $M$-generic for $\P$ and let $H$ be $N$-generic for $j(\P)$. Then the following are equivalent. 
\begin{enumerate}
\item $j"G\subseteq H$
\item There exists an elementary embedding $j^*:M[G]\to N[H]$, such that $j^*(G)=H$ and $j^*\restrict M =j$.
\end{enumerate}
\end{lemma}
\noindent The embedding $j^*$ in condition (2) above is called a $\emph{lift}$ of $j$.

Suppose $j:V\to M$ is an elementary embedding. A set $S\in V$ is said to \emph{generate $j$ over $V$} if $M$ is of the form 
\begin{align}
M&=\{j(h)(s)\mid h:[A]^{<\omega}\to V, s\in [S]^{<\omega}, h\in V\}.\label{seedrep}
\end{align}
where $A\in V$ and $S\subseteq j(A)$. In this context, the elements of $S$ are called seeds. For more on `seed theory' and its applications, see \cite{Hamkins:CanonicalSeedsAndPrickryTrees}. I will often make use of the following lemma which states that the above representation (\ref{seedrep}) of the target model of an elementary embedding remains valid after forcing.

\begin{lemma}\label{lemmaseedpreservation}
If $j:V\to M$ is an elementary embedding generated over $V$ by a set $S\in V$ then any lift of this embedding to a forcing extension $j^*:V[G]\to M[j^*(G)]$ is generated by $S$ over $V[G]$ even if $j^*$ is a class in some further forcing extension $N\supseteq V[G]$.
\end{lemma}

The following standard lemma, which appears in \cite[Section 1.2]{Cummings:AModelInWhichGCH}, asserts that embeddings witnessed by extenders are preserved by highly distributive forcing.

\begin{lemma}\label{lemmalambdadist}



If $j:V\to M$ is generated by $S\subseteq j(I)$, and $V[G]$ is obtained by ${\leq}|I|$-distributive forcing, then $j$ lifts uniquely to an embedding $j:V[G]\to M[j(G)]$.

\end{lemma}

\begin{proof}
Suppose $\P$ is $\leq|I|$-distributive forcing and that $G$ is $V$-generic for $\P$. By intersecting at most $|I|$ open dense subsets of $\P$, one may show that $j"G$ generates an $M$-generic filter on $j(\P)$.
\end{proof}

\subsection{Iterations of Almost Homogeneous Forcing}\label{sectionhomogeneousiteration}



In the course of proving Theorem \ref{theoremwoodin}, the next lemma will be used to show that a certain forcing iteration is almost homogeneous. Recall that a poset $\P$ is \emph{almost homogeneous} if for each pair of conditions, $p,q\in\P$, there is an automorphism $f\in\Aut(\P)$ such that $f(p)$ and $q$ are compatible. If $\P$ is an almost homogeneous forcing notion, a $\P$-name $\dot{x}$ is called \emph{symmetric} if for every automorphism $f\in\Aut(\P)$ one has $\forced_\P f(\dot{x})=\dot{x}$, where $f(\dot{x})$ denotes the $\P$-name obtained from $\dot{x}$ by recursively applying $f$. 

\begin{lemma}\label{lemmahomogeneousiteration}
Suppose $\P_\beta=\langle (\P_\alpha,\dot{\Q}_\alpha)\mid\alpha<\beta\rangle$ is an Easton support iteration and that for each $\alpha<\beta$ one has $\forces_{\P_\alpha}$ ``$\dot{\Q}_\alpha$ is almost homogeneous.'' Suppose further that for each $\alpha<\beta$, one has that $\dot{\Q}_\alpha$ is a symmetric $\P_\alpha$-name; that is, for each automorphism $f\in\Aut(\P_\alpha)$ one has $\forced_{\P_\alpha}f(\dot{\Q}_\alpha)=\dot{\Q}_\alpha$. Then the iteration $\P_\beta$ is almost homogeneous.
\end{lemma}

For a proof of Lemma \ref{lemmahomogeneousiteration} see \cite[Lemma 4]{DobrinenFriedman:HomogeneousIteration}.

\subsection{Some facts concerning Woodin cardinals}\label{sectionwoodin}

Woodin cardinals were originally formulated, by Woodin, for the purpose of reducing the large cardinal consistency strength needed for obtaining a model of the theory ``every set of reals in $L(\R)$ is Lebesgue measurable'' (see the discussion around Theorem 32.9 in \cite{Kanamori:Book}). Although part of the folklore, there has been little published, to the author's knowledge, concerning the preservation of Woodin cardinals through forcing. For example, it is widely known that if $\delta$ is a Woodin cardinal, then the following forcing notions preserve this: (1) any forcing of size less than $\delta$ (see \cite{HamkinsWoodin:SmallForcing} for this result and more), (2) the canonical forcing to achieve $\GCH$, and (3) any ${<}\delta$-closed forcing (see Lemma \ref{lemmawoodinclosed} below). 

I now give some further definitions and lemmas that will be used in the proof of Theorem \ref{theoremwoodin}. The following definition is due to Woodin.

\begin{definition}\label{definitionwoodin}
A cardinal $\delta$ is called a \emph{Woodin cardinal} if for every function $f:\delta\to\delta$ there is a $\kappa<\delta$ with $f"\kappa\subseteq\kappa$ and there is a $j:V\to M$ with critical point $\kappa$ such that $V_{j(f)(\kappa)}\subseteq M$.
\end{definition}

As it turns out, Woodin cardinals have another characterization which is more commonly used in practice. We present several versions of this characterization in the next lemma. First let me give a few definitions. Suppose $A\subseteq V_\delta$ and $\kappa<\delta$. One says that $\kappa$ is \emph{$\gamma$-strong for $A$} if there is a $j:V\to M$ with critical point $\kappa$ such that $V_\gamma\subseteq M$, $j(\kappa)>\gamma$, and $j(A)\cap V_\gamma= A\cap V_\gamma$. By definition $\kappa$ is \emph{${<}\delta$-strong for $A$} if $\kappa$ is $\gamma$-strong for $A$ for each $\gamma<\delta$.

\begin{lemma}\label{lemmawoodin}
The following are equivalent.
\begin{enumerate}
\item[$(1)$]\label{l1} $\delta$ is a Woodin cardinal. 
\item[$(2)$] For every $A\subseteq V_\delta$ the following set is stationary. $$\{\kappa<\delta\mid\textrm{$\kappa$ is ${<}\delta$-strong for $A$}\}$$
\item[$(3)$] For every $A\subseteq V_\delta$ there is a $\kappa<\delta$ that is ${<}\delta$-strong for $A$.
\item[$(4)$] For every $A\subseteq \delta$ there is a $\kappa<\delta$ such that for any $\gamma<\delta$ there is a $j:V\to M$ with critical point $\kappa$ such that $\gamma<j(\kappa)$ and $j(A)\cap\gamma =A\cap\gamma$.
\item[$(5)$] For any pair of sets $A_0,A_1\subseteq\delta$ there is a $\kappa<\delta$ such that for any $\gamma<\delta$ there is a $j:V\to M$ with critical point $\kappa$ such that $\gamma<j(\kappa)$, $j(A_0)\cap\gamma= A_0\cap\gamma$, and $j(A_1)\cap\gamma = A_1\cap\gamma$. 


\end{enumerate}
\end{lemma}





For a proof that (1), (2), and (3) are equivalent one may see \cite[Theorem 26.14]{Kanamori:Book}. To see that (4) and (5) are equivalent to (3) one just needs to use standard coding techniques (details are worked out in \cite{Cody:Dissertation}).

If $\delta$ is a Woodin cardinal, then this is witnessed by embeddings as in Lemma \ref{lemmawoodin}(3). By considering a factor diagram, these embeddings can always be assumed to be extender embeddings (see \cite{HamkinsWoodin:SmallForcing}), meaning that the target of such an embedding, $j:V\to M$, is of the form 
$$M=\{j(h)(a)\mid \textrm{$h:V_\kappa\to V$, $a\in V_\gamma$, and $h\in V$}\}.$$

The following lemma will be required in our proof of Theorem \ref{theoremwoodin}. 

\begin{lemma}\label{lemmawoodinmenas}
Suppose $\kappa$ is ${<}\delta$-strong for $A\subseteq V_\delta$ where $\delta$ is a Woodin cardinal. There is a function $\ell:\kappa\to \kappa$ such that for any $\theta<\delta$ there is a $j:V\to M$ witnessing that $\kappa$ is $\theta$-strong for $A$ such that $j(\ell)(\kappa)=\theta$.
\end{lemma}



\begin{proof}

Define a function $\ell$ with domain $\kappa$ as follows. If $\gamma<\kappa$ is not ${<}\delta$-strong for $A$ then define $\ell(\gamma)$ to be the least ordinal such that $\gamma$ is not $\ell(\gamma)$-strong for $A$. Otherwise define $\ell(\gamma)=0$. 

Let me show that $\ell(\gamma)<\kappa$ for each $\gamma<\kappa$. Suppose $\gamma$ is not ${<}\delta$-strong for $A$ and that $\ell(\gamma)\geq\kappa$. I will show that since $\kappa$ is ${<}\delta$-strong for $A$ it follows that $\gamma$ is also ${<}\delta$-strong for $A$, a contradiction. Choose $\theta<\delta$ and let $j:V\to M$ witness that $\kappa$ is $\theta$-strong for $A$. Since $\ell(\gamma)\geq\kappa$ it follows that $\gamma$ is ${<}\kappa$-strong for $A$. By elementarity $\gamma=j(\gamma)$ is ${<}j(\kappa)$-strong for $j(A)$ in $M$. Thus $\gamma$ is $\theta$-strong for $j(A)$ in $M$. Let $i:M\to N$ witness this. Now let $j^*:=i\circ j:V\to N$. It follows that $\gamma$ is the critical point of $j^*$, that $j^*(\gamma)=i(j(\gamma))=i(\gamma)>\theta$, and $j^*(A)\cap\theta = i(j(A))\cap\theta = j(A)\cap\theta = A\cap\theta$. Hence $\gamma$ is $\theta$-strong for $A$. This implies that $\gamma$ is ${<}\delta$-strong for $A$, a contradiction. This shows that $\ell$ is a function from $\kappa$ to $\kappa$.

Now fix $\theta<\delta$ and let $j:V\to M$ be an embedding witnessing that $\kappa$ is $\theta$-strong for $A$ such that $\kappa$ is not $\theta$-strong for $A$ in $M$. Such an embedding can be obtained by taking $j(\kappa)$ to be minimal. It follows that $\kappa$ is $\beta$-strong for $A$ in $M$ for every $\beta<\theta$. Thus, $j(f)(\kappa)=\theta$.
\end{proof}

The next widely known lemma is important for our proof of Theorem \ref{theoremwoodin}, because it easily implies that if $\delta$ is a Woodin cardinal, then one can force the continuum function to agree with any Easton function on the interval $[\delta,\infty)$.

\begin{lemma}\label{lemmawoodinclosed}
If $\delta$ is a Woodin cardinal and $\P$ is ${<}\delta$-closed then $\delta$ remains Woodin after forcing with $\P$.\footnote{I would like to thank Arthur Apter for an enlightening discussion concerning Lemma \ref{lemmawoodinclosed} and its proof.}
\end{lemma}

\begin{proof} For this proof, I will use the definition of Woodin cardinal as opposed to one of the characterizations given in Lemma \ref{lemmawoodin}.
Let $G$ be generic for $\P$ and suppose $p\in G$ and $p\forces\dot{f}:\delta\to\delta$. Let $D$ be the set of conditions $q\leq p$ such that $q$ forces there is a $\kappa<\delta$ such that $\dot{f}"\kappa\subseteq\kappa$ and there is a $j:V[\dot{G}]\to M[j(\dot{G})]$ with critical point $\kappa$ and $(V_{j(\dot{f})(\kappa)})^{V[\dot{G}]}\subseteq M[j(\dot{G})]$. Note that the existence of the previous embedding is equivalent to the existence of an extender that has a first order definition. I will show that $D$ is dense below $p$. Choose $r\leq p$ and use the ${<}\delta$-closure of $\P$ to find a descending sequence $\langle p_\alpha\mid\alpha<\delta\rangle$ of conditions below $r$ such that $p_\alpha$ decides $\dot{f}\restrict(\alpha+1)$ for each $\alpha<\delta$. Let $F:\delta\to\delta$ be the function in $V$ determined by the sequence $\langle p_\alpha\mid\alpha<\delta\rangle$. By applying the Woodinness of $\delta$ in $V$ to $F$ find a $\kappa<\delta$ such that $F"\kappa\subseteq \kappa$ and there is a $j:V\to M$ with critical point $\kappa$ and $V_{j(F)(\kappa)}\subseteq M$. In addition, by taking a factor embedding if necessary, one may assume without loss of generality that $M=\{j(h)(a)\mid \textrm{$h:V_\kappa\to V$, $a\in V_{j(F)(\kappa)}$, and $h\in V$}\}$. Now choose $\alpha<\delta$ large enough so that $p_\alpha$ forces $\dot{f}$ to agree with $F$ up to and including at $\kappa$. Let $H$ be $V$-generic for $\P$ with $p_\alpha\in H$. Then $\dot{f}^H"\kappa\subseteq\kappa$. Since $\P$ is ${\leq}\kappa$-distributive, it follows by Lemma \ref{lemmalambdadist} that $j$ lifts to $j:V[H]\to M[j(H)]$. By elementarity and the fact that $p_\alpha\in H$, it follows that $j(\dot{f}^H)(\kappa)=j(F)(\kappa)$. Since $\P$ is ${<}\delta$-closed, it follows that $(V_{j(F)(\kappa)})^{V[H]}=V_{j(F)(\kappa)}$. Thus, $(V_{j(\dot{f}^H)(\kappa)})^{V[H]}=(V_{j(F)(\kappa)})^{V[H]}=V_{j(F)(\kappa)}\subseteq M\subseteq M[j(H)]$. This shows that $p_\alpha\in D$ and thus that $D$ is dense below $p$.

Now choose a condition $q\in G\cap D$ so that by the definition of $D$ it follows that in $V[G]$ there is a $\kappa<\delta$ such that $f"\kappa\subseteq\kappa$ and there is a $j:V[G]\to M[j(G)]$ with critical point $\kappa$ and $V[G]_{j(f)(\kappa)}\subseteq M[j(G)]$.
\end{proof}

\subsection{Sacks forcing on uncountable cardinals}

Kanamori gave a definition for a version of Sacks forcing on uncountable cardinals in \cite{Kanamori:PerfectSetForcing}. In what follows, I will use a definition of Sacks forcing on inaccessible cardinals given by Friedman and Thompson in \cite{FriedmanThompson:PerfectTreesAndElementaryEmbeddings} (and used in \cite{FriedmanHonzik:EastonsTheoremAndLargeCardinals}), which works particularly well for preserving large cardinals; for the reader's convenience, I will recall the definition and some basic properties of this forcing.

Suppose $\kappa$ is an inaccessible cardinal. Then $p\subseteq 2^{<\kappa} $ is a \emph{perfect $\kappa$-tree} if the following conditions hold.
\begin{enumerate}
\item[$(1)$] If $s\in p$ and $t \in 2^{<\kappa}$ is an initial segment of $s$, then $t\in p$.
\item[$(2)$] If $\langle s_\alpha\mid\alpha<\eta\rangle$ is a sequence of elements of $p$ with $\eta<\kappa$ where $s_\alpha\subseteq s_\beta$ for $\alpha<\beta$, then $\bigcup_{\alpha<\eta} s_\alpha\in p$.
\item[$(3)$] For each $s\in p$ there is a $t\in p$ with $s\subseteq t$ and $t\concat 0,t\concat 1\in p$.
\item[$(4)$] Let $\Split(p)=\{s\in p\mid s\concat 0,s\concat 1\in p\}$. Then for some unique closed unbounded set $C(p)\subseteq\kappa$, $\Split(p)=\{s\in p\mid \length(s)\in C(p)\}$.
\end{enumerate}

\emph{Sacks forcing on $\kappa$} is denoted by $\Sacks(\kappa)$ and conditions in $\Sacks(\kappa)$ are perfect $\kappa$-trees. For $p,q\in\Sacks(\kappa)$, one says that $p$ is stronger than $q$ and writes $p\leq q$ if and only if $p\subseteq q$. For a condition $p\in\Sacks(\kappa)$ let $\langle \alpha_i\mid i<\kappa\rangle$ be the increasing enumeration of $C(p)$. Let $\Split_i(p):=\{s\in p \mid \length(s)=\alpha_i\}$ denote the $i^{th}$ splitting level of $p$. For $p,q\in\Sacks(\kappa)$, define $p\leq_\beta q$ if and only if $p\leq q$ and $\Split_i(p)=\Split_i(q)$ for $i<\beta$. It is easy to verify that $\Sacks(\kappa)$ is ${<}\kappa$-closed and satisfies the $\kappa^{++}$-chain condition under $\GCH$. By standard arguments, this implies that $\Sacks(\kappa)$ preserves cardinals less than or equal to $\kappa$ and greater than or equal to $\kappa^{++}$ under $\GCH$. Furthermore, as shown in \cite{FriedmanThompson:PerfectTreesAndElementaryEmbeddings}, $\Sacks(\kappa)$ satisfies the following fusion property. If $\langle p_\alpha\mid\alpha<\kappa\rangle$ is a decreasing sequence of conditions in $\Sacks(\kappa)$ and for each $\alpha<\kappa$, $p_{\alpha+1}\leq_{\alpha}p_{\alpha}$, then the sequence has a lower bound in $\Sacks(\kappa)$. The sequence $\langle p_\alpha\mid\alpha<\kappa\rangle$ is called a fusion sequence. This fusion property implies that $\Sacks(\kappa)$ preserves $\kappa^+$ by the following straightforward argument. Suppose $p\forces \dot{f}:\check{\kappa}\to\check{\kappa}^+$. One can build a fusion sequence $\langle p_\alpha\mid\alpha<\kappa\rangle$ such that for each $\alpha<\kappa$, the condition $p_\alpha\in\Sacks(\kappa)$ forces $\dot{f}(\check{\alpha})$ to equal the check name of an element of some set $A_\alpha=\{{\beta}_\xi\mid\xi<2^\alpha\}$ where each $\beta_\xi$ is less than $\kappa^+$. By the fusion property, this sequence has a lower bound, call it $r$, and it follows that $r\forces \ran(\dot{f})\subseteq \bigcup_{\alpha<\kappa}A_\alpha$. Since $\bigcup_{\alpha<\kappa}A_\alpha$ has size at most $\kappa$, it follows that $r$ forces $\ran(\dot{f})$ to be bounded below $\kappa^+$. The forcing $\Sacks(\kappa)$ adds a single subset of $\kappa$ given by a cofinal branch through $2^{<\kappa}$ and preserves cardinals under $\GCH$.

Define $\Sacks(\kappa,\lambda)$ to be the product forcing obtained by taking the product of $\lambda$-many copies of $\Sacks(\kappa)$ with supports of size less than or equal to $\kappa$. Thus, a condition $\vec{p}\in\Sacks(\kappa,\lambda)$ can be thought of as a function $\vec{p}:\lambda\to\Sacks(\kappa)$ such that the set $\{\alpha<\lambda\mid \vec{p}(\alpha)\not= 2^{<\kappa}\}$ has size at most $\kappa$. The ordering on $\Sacks(\kappa,\lambda)$ is given by the usual product ordering. It is easy to verify that $\Sacks(\kappa,\lambda)$ is ${<}\kappa$-closed and satisfies the $\kappa^{++}$-chain condition under $\GCH$. Thus, assuming $\GCH$, the poset $\Sacks(\kappa,\lambda)$ preserves cardinals less than or equal to $\kappa$ and greater than or equal to $\kappa^{++}$. To show that $\Sacks(\kappa,\lambda)$ preserves $\kappa^{+}$ one may use the following generalized fusion property (see \cite{FriedmanThompson:PerfectTreesAndElementaryEmbeddings}). For $X\subseteq\lambda$ and $\vec{p},\vec{q}\in\Sacks(\kappa,\lambda)$ write $\vec{p}\leq_{\beta,X}\vec{q}$ if and only if $\vec{p}\leq \vec{q}$ and for each $\alpha\in X$, $\vec{p}(\alpha)\leq_\beta\vec{q}(\alpha)$. The generalized fusion property for $\Sacks(\kappa,\lambda)$ asserts that if $\langle \vec{p}_\alpha\mid\alpha<\kappa\rangle$ is a descending sequence of conditions in $\Sacks(\kappa,\lambda)$ and there is an increasing sequence $\langle X_\alpha\mid\alpha<\kappa\rangle$ of subsets of $\lambda$, each of size less than $\kappa$, such that $\bigcup_{\alpha<\kappa}X_\alpha=\bigcup_{\alpha<\kappa}\supp(\vec{p}_\alpha)$, and for each $\beta<\kappa$, $\vec{p}_{\beta+1}\leq_{\beta,X_\beta}\vec{p}_{\beta}$, then there is a lower bound of the sequence $\langle \vec{p}_\alpha\mid\alpha<\kappa\rangle$ in $\Sacks(\kappa,\lambda)$. The above generalized fusion property implies that $\kappa^+$ is preserved by the following argument. Suppose $\vec{p}\forces \dot{f}:\check{\kappa}\to\check{\kappa}^+$. One can build a fusion sequence $\langle \vec{p}_\alpha\mid\alpha<\kappa\rangle$ such that for each $\alpha<\kappa$, the condition $\vec{p}_\alpha$ forces $\dot{f}(\alpha)$ to belong to a subset of $\kappa^+$ of size $(2^\alpha)^\gamma$ for some $\gamma<\kappa$. A lower bound $\vec{r}$ of this fusion sequence forces a bound on $f$ below $\kappa^+$.

Since $\Sacks(\kappa,\lambda)$ is not $\kappa^+$-c.c. more than Lemma \ref{lemmachain} will be required to see that $\Sacks(\kappa,\lambda)$ preserves closure under $\kappa$ sequences on inner models. For this reason will need the following.
\begin{lemma}\label{lemmaclosuresacks}
Suppose $M\subseteq V$ is an inner model with $M^\kappa\subseteq M$ in $V$. If $G$ is $V$-generic for $\Sacks(\kappa,\lambda)$, then $M[G]^\kappa\subseteq M[G]$ in $V[G]$.
\end{lemma}
\begin{proof}
Let me recall the proof given in \cite[Lemma 3]{FriedmanThompson:PerfectTreesAndElementaryEmbeddings}. Let $G$ be generic for $\Sacks(\kappa,\lambda)$. Suppose $X$ is a $\kappa$-sequence of ordinals in $V[G]$ and that this is forced by $p\in G$. Using generalized fusion, one can show that every $q\leq p$ can be extended to a condition $r$ such that $r$ forces that $X$ can be determined from $r$ and $G$. This implies that there is such an $r\in G$. Since $r$ and $G$ are both in $M[G]$, it follows that $X\in M[G]$. 
\end{proof}


Easton's Lemma states that if $\P$ and $\Q$ are forcing notions where $\P$ is $\kappa^+$-c.c. and $\Q$ is ${\leq}\kappa$-closed, then $\forces_\P$ ``$\Q$ is ${\leq}\kappa$-distributive.'' The following lemma, which is analagous to Easton's Lemma, will be important for the proof of our main theorem. 
\begin{lemma}\label{lemmaeastonforsacks}
Suppose $\P$ is any ${\leq}\kappa$-closed forcing and $\alpha$ is an ordinal. Then after forcing with $\Sacks(\kappa,\alpha)$, $\P$ remains ${\leq}\kappa$-distributive.
\end{lemma}

\begin{proof}
Suppose $p\in\Sacks(\kappa,\lambda)\times\P$ forces that $\dot{f}$ is a function with $\dom(\dot{f})=\kappa$. One can show, using generalized fusion in the first coordinate and closure in the second coordinate, that every condition $q$ below $p$ can be extended to a condition $r$ which forces over $\Sacks(\kappa,\lambda)\times\P$ that the values of $\dot{f}$ can be determined from $r$ and $G$, the generic for $\Sacks(\kappa,\lambda)$. 
\end{proof}



For a more detailed proof of Lemma \ref{lemmaeastonforsacks} see \cite[Lemma 3.7]{FriedmanHonzik:EastonsTheoremAndLargeCardinals}.

%
%

\section{Proof of Theorem \ref{theoremwoodin}}

Recall the statement of Theorem \ref{theoremwoodin}.

\begin{theorem31}
Suppose $\GCH$ holds, $F:\REG\to\CARD$ is an Easton function, and $\delta$ is a Woodin cardinal with $F"\delta\subseteq\delta$. Then there is a cofinality-preserving forcing extension in which $\delta$ remains Woodin and $2^\gamma=F(\gamma)$ for each regular cardinal $\gamma$.
\end{theorem31}

\noindent \textit{Proof.}

Suppose $\delta$ is a Woodin cardinal and $F:\REG\to\CARD$ is an Easton function with $F"\delta\subseteq\delta$. For an ordinal $\alpha$ let $\bar{\alpha}$ denote the least closure point of $F$ greater than $\alpha$. For a regular cardinal $\gamma$, the notation $\Add(\gamma,F(\gamma))$ denotes the poset for adding $F(\gamma)$ Cohen subsets to $\gamma$. The forcing is, the same iteration introduced in \cite{FriedmanHonzik:EastonsTheoremAndLargeCardinals}, that is, an Easton support iteration $\P=\langle (\P_{\eta},\dot{\Q}_{\eta}) : \eta\in\ORD \rangle$ of Easton support products defined as follows. 
\begin{enumerate}
\item[$(1)$] If $\eta$ is an inaccessible closure point of $F$ in $V^{\P_{\eta}}$, then $\dot{\Q}_{\eta}$ is a $\P_{\eta}$-name for the Easton support product $$\Sacks(\eta,F(\eta))\times \prod_{\gamma\in(\eta,\bar{\eta})\cap\REG} \Add(\gamma,F(\gamma))$$
as defined in $V^{\P_{\eta}}$ and $\P_{\eta+1}=\P_{\eta}*\dot{\Q}_{\eta}$
\item[$(2)$] If $\eta$ is a singular closure point of $F$ in $V^{\P_{\eta}}$, then $\dot{\Q}_{\eta}$ is a $\P_{\eta}$-name for $\prod_{\gamma\in[\eta,\bar{\eta})\cap\REG} \Add(\gamma,F(\gamma))$  as defined in $V^{\P_{\eta}}$ and $\P_{\eta+1}=\P_{\eta}*\dot{\Q}_{\eta}$.
\item[$(3)$] Otherwise, if $\eta$ is not a closure point of $F$, then $\dot{\Q}_\eta$ is a $\P_\eta$-name for trivial forcing and $\P_{\eta+1}=\P_\eta*\dot{\Q}_\eta$.
\end{enumerate}

Let $G$ be $V$-generic for $\P$. As in \cite{FriedmanHonzik:EastonsTheoremAndLargeCardinals}, it follows that cardinals are preserved (see \cite[Lemma 3.6]{FriedmanHonzik:EastonsTheoremAndLargeCardinals}) and that for each regular cardinal $\gamma$ one has $2^\gamma=F(\gamma)$ (see \cite[Theorem 3.8]{FriedmanHonzik:EastonsTheoremAndLargeCardinals}). 

Let me now discuss some notation that will be useful for factoring $\P$. If $\eta$ is a closure point of $F$, then one can factor $\P\cong\P_\eta * \dot{\P}_{[\eta,\infty)}$ where $\P_\eta$ denotes the iteration up to stage $\eta$ and $\dot{\P}_{[\eta,\infty)}$ is a $\P_\eta$-name for the remaining stages. Thus $G$ naturally factors as $G\cong G_\eta*G_{[\eta,\infty)}$. The stage $\eta$ forcing in the iteration $\P$ is $\Q_{\eta}$ and I will write $\Q_{\eta}=\Q_{[\eta,\bar{\eta})}$ to emphasize the interval on which the stage $\eta$ forcing has an effect. Let $H_{[\eta,\bar{\eta})}$ denote the $V[G_{\eta}]$-generic for $\Q_{[\eta,\bar{\eta})}$ obtained from $G$. Let $\R_\gamma$ denote a particular factor of the product forcing $\Q_{[\eta,\bar{\eta})}$ so that $\Q_{[\eta,\bar{\eta})}=\prod_{\gamma\in[\eta,\bar{\eta})\cap\REG}\R_\gamma$. In this situation let $H_\gamma$ denote that $V[G_{\eta}]$-generic for $\R_\gamma$ obtained from $G$. In general, if $I\subseteq [\eta,\bar{\eta})$ then let $\Q_I=\prod_{\gamma\in I\cap\REG}\R_\gamma$.

Since $\P_{[\delta,\infty)}$ is ${<}\delta$-closed in $V^{\P_\delta}$, it follows by Lemma \ref{lemmawoodinclosed} that if $\delta$ is Woodin in $V^{\P_\delta}$ then $\delta$ remains Woodin in $V^{\P_\delta*\dot{\P}_{[\delta,\infty)}}$. Thus it will suffice to show that $\delta$ remains Woodin in $V[G_\delta]$. Let me note here that by the previous statements, one could have defined the iteration above so that $\dot{\P}_{[\delta,\infty)}$ is simply a $\P_\delta$-name for an Easton support product of Cohen forcing.

 In what follows I will use the fact that since conditions in $\P_\delta$ have bounded support, one can view them as sequences of length less than $\delta$. Indeed, by cutting off trivial coordinates, one can view a condition $p\in\P_\delta$ as being a condition in some initial segment of the poset.

I will show that property (3) in Lemma \ref{lemmawoodin} holds in $V[G_\delta]$. Suppose $A\subseteq \delta$ with $A\in V[G_\delta]$ and let $\dot{A}$ be a $\P_\delta$-name for $A$. For each $\alpha<\delta$, let $A_\alpha$ be a maximal antichain of conditions in $\P_\delta$ that decide $\check{\alpha}\in\dot{A}$. Define a function $u:\delta\to\delta$ such that 
\begin{align*}
u(\gamma)=\ &\textrm{``the least ordinal $\beta$ such that for each $\alpha<\gamma$,} \\
     &\ \textrm{the antichain $A_\alpha$ is contained in $\P_\beta$''}
\end{align*}
The value of $u(\gamma)$ indicates how much of the generic filter is required to correctly evaluate the name $\dot{A}$ up to $\gamma$. 


Now I will apply the Woodinness of $\delta$ in $V$. By an argument similar to that for Lemma \ref{lemmawoodin}(5), i.e. by coding the name $\dot{A}\subseteq V_\delta$, the Easton funciton $F\cap\delta\times\delta$, and the function $u\subseteq\delta\times\delta$, into a single subset of $\delta$, it follows that there is a $\kappa<\delta$ that is ${<}\delta$-strong for the name $\dot{A}$, the Easton function $F\restrict \delta$, and the function $u$. As an abbreviation, I will say that such a $\kappa$ is ${<}\delta$-strong for $\langle\dot{A},F,u\rangle$. Since $C_F:=\{\alpha<\delta\mid F"\alpha\subseteq\alpha\}$ is a closed unbounded subset of $\delta$ and since the set $S:=\{\kappa<\delta\mid\textrm{$\kappa$ is ${<}\delta$-strong for $\langle \dot{A},F,u\rangle$}\}$ is stationary, one may choose such a $\kappa\in C\cap S$. This is, of course, necessary since there is no hope of $\kappa$ remaining measurable in $V[G_\delta]$ if $\kappa$ is not a closure point of $F$. 

Fix $\kappa<\delta$ such that $\kappa$ is a closure point of $F$ and $\kappa$ is ${<}\delta$-strong for $\langle \dot{A},F,u\rangle$. Fix a function $\ell:\kappa\to\kappa$ as in Lemma \ref{lemmawoodinmenas}. I will show that property (3) in Lemma \ref{lemmawoodin} holds for this $\kappa$ and the initially chosen $A\subseteq \delta$ in $V[G_\delta]$. 

Since the inaccessible closure points of $F$ are unbounded in $\delta$, I may choose $\mu$ to be an inaccessible closure point of $F$ with $F(\kappa)<\mu<\delta$. It will suffice to show that in $V[G_\delta]$ there is an embedding $j:V[G_\delta]\to M[j(G_\delta)]$ with critical point $\kappa$ and $j(A)\cap\mu = A\cap\mu$. Now I will define a singular $\theta>\mu$ and lift an embedding that is $\theta$-strong for $\langle\dot{A},F,u\rangle$. I will also show that the lifted embedding satisfies $j(A)\cap \mu=A\cap\mu$. Using a singular degree of strength is advantageous since this will mean there will be no forcing over $\theta$, and it will follow that the relevant tail forcing will be sufficiently closed. Let $\mu'$ be the least inaccessible closure point of $u$ greater than $\mu$. Define a sequence $\langle\gamma_\alpha\mid\alpha<\kappa^+\rangle$ by recursion as follows. Let $\gamma_0$ be the least inaccessible closure point of $F$ greater than $\mu'$. Assuming $\gamma_\alpha$ is defined where $\alpha<\kappa^+$, let $\gamma_{\alpha+1}$ be the least inaccessible closure point of $F$ greater than $\gamma_\alpha$. At limit stages $\zeta<\kappa^+$, assuming $\langle \gamma_\alpha\mid\alpha<\zeta\rangle$ is defined, let $\gamma_\zeta$ be the least inaccessible closure point of $F$ greater than $\sup\{\gamma_\alpha\mid\alpha<\zeta\}$. Now define $\theta:=\sup\{\gamma_\alpha\mid\alpha<\kappa^+\}$. We have
$$\kappa<F(\kappa)<\mu<\mu'<\gamma_0<\cdots<\gamma_\alpha<\cdots<\theta.$$
For emphasis, let me state the following explicitly.
\begin{itemize}
\item $\langle\gamma_\alpha\mid\alpha<\kappa^+\rangle$ is a discontinuous sequence of inaccessible closure points of $F$.
\item $\theta=\sup\{\gamma_\alpha\mid\alpha<\kappa^+\}$
\item $u"\mu'\subseteq\mu'$
\end{itemize}

By assumption on $\kappa$, there is a $j:V\to M$ with critical point $\kappa$ such that the following hold.
\begin{enumerate}
\item[(1)] $V_\theta\subseteq M$ $\and$ $\theta<j(\kappa)$
\item[(2)] $j(\dot{A})\cap\theta=\dot{A}\cap\theta$ $\and$ $j(F)\restrict\theta=F\restrict\theta$ $\and$ $j(u)\restrict\theta=u\restrict\theta$
\item[(3)] $M=\{j(h)(s)\mid h:V_\kappa\to V, s\in V_\theta, h\in V\}$
\item[(4)] $j(\ell)(\kappa)=\theta$ (using Lemma \ref{lemmawoodinmenas})
\end{enumerate}
Since $j(F)\restrict\theta=F\restrict\theta$, the sequence $\langle \gamma_\alpha\mid\alpha<\kappa^+\rangle$ can be constructed in $M$ from $j(F)$ just as it was constructed in $V$ from $F$. This implies that
\begin{enumerate}
\item[(5)] $\cf(\theta)^M=\kappa^+$.
\end{enumerate}
\noindent Property (5) will be important because it ensures that there is no forcing over $\theta$ in the iteration $j(\P_\delta)$.

\subsection{Lifting $j$ Through $G_\kappa$.}\label{subsectionwoodingkappa}

In order to lift $j$ to $V[G_\kappa]$, I will find an $M$-generic filter $j(G_\kappa)$ for $j(\P_\kappa)$ that satisfies $j"G_\kappa\subseteq j(G_\kappa)$. To do so, the length $j(\kappa)$ iteration $j(\P_\kappa)$ will be factored in $M$. Since $V_\theta\subseteq M$ it follows that $j(\P_\kappa)\cong \P_{\gamma_0}*\dot{\widetilde{\P}}_{[\gamma_0,\theta)}*\dot{\widetilde{\P}}_{[\theta,j(\kappa))}$ where $\dot{\widetilde{\P}}_{[\gamma_0,\theta)}$ is a $\P_{\gamma_0}$-term for the iteration over the interval $[\gamma_0,\theta)$ as defined in $M^{\P_{\gamma_0}}$ and similarly $\dot{\widetilde{\P}}_{[\theta,j(\kappa))}$ is a $\P_{\gamma_0}*\dot{\widetilde{\P}}_{[\gamma_0,\theta))}$-term for the tail of the iteration $j(\P_\kappa)$ as defined in $M^{\P_{\gamma_0}*\dot{\widetilde{\P}}_{[\gamma_0,\theta)}}$. Since $V_\theta\subseteq M$, the iteration $j(\P_\kappa)$ agrees with $\P_\delta$ up to stage $\gamma_0$. Thus it follows that $G_{\gamma_0}$ is $M$-generic for $\P_{\gamma_0}$. Since $\theta$ is singular in $V$, conditions in $\P_{[\gamma_0,\theta)}$ are allowed to have unbounded support. Since $M$ and $V$ do not agree on the collection of unbounded subsets of $\theta$, it follows by a density argument that $G_{[\gamma_0,\theta)}$ is not contained in $\widetilde{\P}_{[\gamma_0,\theta)}$. Nonetheless, Lemmas \ref{lemmapinfty} and \ref{lemmaaut} below will establish that there is an $M[G_{\gamma_0}]$-generic filter, call it $\widetilde{G}_{[\gamma_0,\theta)}$, in $V[G_{\gamma_0}][G_{[\gamma_0,\theta)}]$ for $\widetilde{\P}_{[\gamma_0,\theta)}$. In Lemma \ref{lemmapinfty}, I will show that there is a condition $p_\infty\in \P_{[\gamma_0,\theta)}$ which forces all dense subsets of $\widetilde{\P}_{[\gamma_0,\theta)}$ in $M[G_{\gamma_0}]$ to be met by $G_{[\gamma_0,\theta)}$. It might not be the case that $p_\infty\in G_{[\gamma_0,\theta)}$, but in Lemma \ref{lemmaaut}, I will show that $p_\infty$ is in an automorphic image of $G_{[\gamma_0,\theta)}$, which I shall argue is good enough.

Let me note here that the proof of Lemma \ref{lemmapinfty} resembles the construction of $p_\infty$ in \cite[Sublemma 3.12]{FriedmanHonzik:EastonsTheoremAndLargeCardinals}. However, there is an important difference in that the forcing here, namely $\P_{[\gamma_0,\theta)}$, is an iteration involving generalize Sacks forcing, whereas in \cite{FriedmanHonzik:EastonsTheoremAndLargeCardinals}, the analagous forcing is a product of Cohen forcing.

\begin{lemma}\label{lemmapinfty}
There is a condition $p_\infty\in \P_{[\gamma_0,\theta)}$ such that if $G^*_{[\gamma_0,\theta)}$ is $V[G_{\gamma_0}]$-generic for $\P_{[\gamma_0,\theta)}$ with $p_\infty\in G^*_{[\gamma_0,\theta)}$, then $G^*_{[\gamma_0,\theta)}\cap \widetilde{\P}_{[\gamma_0,\theta)}$ is $M[G_{\gamma_0}]$-generic for $\widetilde{\P}_{[\gamma_0,\theta)}$.
\end{lemma}

\begin{proof}

By our choice of $\theta$, the sequence $\langle \gamma_\alpha\mid \alpha<\kappa^+\rangle$ is an increasing cofinal sequence of inaccessible closure points of $F$ in $\theta$. Recall the placement of the following ordinals.
$$\mu<\mu'<\gamma_0<\gamma_1<\cdots<\gamma_\alpha<\cdots<\theta$$
It follows that, in $M[G_{\gamma_0}]$, for each $\alpha<\kappa^+$,
$$\widetilde{\P}_{[\gamma_0,\theta)}\cong\P_{[\gamma_0,\gamma_\alpha)}*\dot{\widetilde{\P}}_{[\gamma_\alpha,\theta)}$$
where $\P_{[\gamma_0,\gamma_\alpha)}$ is $\gamma_\alpha^+$-c.c. in $V[G_{\gamma_0}]$ and $\widetilde{\P}_{[\gamma_\alpha,\theta)}$ is forced to be ${<}\gamma_\alpha$-closed.

A few sublemmas will be required.

\begin{sublemma}\label{sublemmapd}
Suppose $p_*=(r_*,\dot{q}_*)\in\R*\dot{\Q}$ and $D\subseteq \R*\dot{\Q}$ is open dense. Then there is an $\R$-name $\dot{q}_D$ such that the following hold.
\begin{enumerate}
\item[$(1)$] $(r_*,\dot{q}_D)\leq(r_*,\dot{q}_*)$
\item[$(2)$] $\bar{D}=\{r\leq r_*\mid(r,\dot{q}_D)\in D\}$ is open dense in $\R$ below $r_*$.
\item[$(3)$] $r_*\forces_\R \exists r\in\dot{G}\ (r,\dot{q}_D)\in D$
\end{enumerate}
\end{sublemma}

\begin{proof} I will work below $(r_*,\dot{q}_*)$. Choose $(r_0,\dot{q}_0)\leq (r_*,\dot{q}_*)$ with $(r_0,\dot{q}_0)\in D$. Let $r_0'\leq r$  with $r_0'\perp r_0$. Now let $(r_1,\dot{q}_1)\leq (r_0',\dot{q}_*)$ with $(r_1,\dot{q}_1)\in D$. Proceed by induction.

If $\alpha$ is a successor ordinal, say $\alpha=\beta+1$, choose $r_\beta'\leq r_*$ with $r_\beta'\perp\{r_\xi\mid\xi\leq\beta\}$. Let $(r_{\beta+1},\dot{q}_{\beta+1})\in D$ with $(r_{\beta+1},\dot{q}_{\beta+1})\leq (r_\beta',\dot{q}_*)$.

If $\alpha$ is a limit ordinal, suppose $\{r_\xi\mid\xi<\alpha\}$ is the antichain of $\R$ constructed so far. Let $r_\alpha''\in\R$ be such that $r''_\alpha\perp\{r_\xi\mid\xi<\alpha\}$. Let $(r_\alpha,\dot{q}_\alpha)\in D$ with $(r_\alpha,\dot{q}_\alpha)\leq (r_\alpha'',\dot{q}_*)$.

The process terminates at some stage $\gamma$ once $A:=\{r_\xi\mid\xi<\gamma\}$ forms a maximal antichain of $\R$ below $r_*$. Let $\dot{q}_D$ be the $\R$-name obtained by mixing the names $\dot{q}_\xi$, defined above, over $A$. In other words, $\dot{q}_D$ has the property that for each $\xi<\gamma$ the condition $r_\xi$ forces $\dot{q}_D=\dot{q}_\xi$.

Let me show that (1) holds. Any generic for $\R$ containing $r_*$ will contain $r_\xi$ for some $\xi<\gamma$. Since $r_\xi\forces\dot{q}_D=\dot{q}_\xi$ and $(r_\xi,\dot{q}_\xi)\leq (r_*,\dot{q}_*)$, it follows that $r_\xi\forces \dot{q}_D=\dot{q}_\xi \leq\dot{q}_*$. Hence $r_*\forces \dot{q}_D\leq\dot{q}_*$.


I will now show that (2) holds. Since $D$ is open it easily follows that $\bar{D}$ is open. Suppose $p\leq r_*$ with $p\in \R$. Since $A$ is a maximal antichain of $\R$ below $r_*$ the condition $p$ is compatible with some $r_\xi\in A$. Thus, let $s\in \R$ with $s\leq r_\xi$ and $s\leq p$. Since $(r_\xi,\dot{q}_\xi)\in D$ and $D$ is open dense, to show that $s\in \bar{D}$ it will suffice to show that $(s,\dot{q}_D)\leq (r_\xi,\dot{q}_\xi)$. This easily follows since $s\leq r_\xi$ and $r_\xi\forces \dot{q}_D=\dot{q}_\xi$ imply that $s\forces \dot{q}_D\leq\dot{q}_\xi$.
\end{proof}

\begin{sublemma}\label{sublemmajoel}
Suppose $q\in\widetilde{\P}_{[\gamma_0,\theta)}$. For all functions $h\in V$ with $\dom(h)=V_\kappa$ and all $\beta<\theta$ there is a $p\leq q$ with $p\in\widetilde{\P}_{[\gamma_0,\theta)}$ such that if $p\in G^*_{[\gamma_0,\theta)}$ is $V[G_{\gamma_0}]$-generic for $\P_{[\gamma_0,\theta)}$, then $G^*_{[\gamma_0,\theta)}$ meets every dense subset of $\widetilde{\P}_{[\gamma_0,\theta)}$ of the form $j(h)(a)^{G_{\gamma_0}}$ where $a\in V_\beta$.
\end{sublemma}

\begin{proof}
Fix $q\in\widetilde{\P}_{[\gamma_0,\theta)}$, a function $h$, and $\beta$ as in the statement of the sublemma. I will obtain the condition $p\leq q$ as a lower bound of a descending sequence of conditions in $\widetilde{\P}_{[\gamma_0,\theta)}$. Since $\langle\gamma_\alpha\mid\alpha<\kappa^+\rangle$ is cofinal in $\theta$, one may choose $\gamma_\alpha>|V_\beta|$. It follows that there is an enumeration $\vec{D}=\langle D^h_\xi\mid\xi<\zeta\rangle$, in $M[G_{\gamma_0}]$, of all dense subsets of $\widetilde{\P}_{[\gamma_0,\theta)}$ of the form $j(h)(a)^{G_{\gamma_0}}$ with $a\in V_\beta$. Clearly one has $\zeta\leq|V_\beta|<\gamma_\alpha$. Factor $\widetilde{\P}_{[\gamma_0,\theta)}$ as $\widetilde{\P}_{[\gamma_0,\theta)}\cong\P_{[\gamma_0,\gamma_{\alpha})} * \dot{\widetilde{\P}}_{[\gamma_{\alpha},\theta)}$. In order to simplify notation, let me define $\R:=\P_{[\gamma_0,\gamma_{\alpha})}$ and $\dot{\Q}:=\dot{\widetilde{\P}}_{[\gamma_{\alpha},\theta)}$, so that $\widetilde{\P}_{[\gamma_0,\theta)}\cong \R*\dot{\Q}$. Note that $\forced_\R$ ``$\dot{\Q}$ is ${<}\gamma_{\alpha}$-closed.'' Since $q\in\widetilde{\P}_{[\gamma_0,\theta)}\cong \R*\dot{\Q}$ one may write $q=(r_*,\dot{q}_*)$ where $r_*=q\restrict[\gamma_0,\gamma_{\alpha})\in\R$ and $\dot{q}_*$ denotes the $\R$-name, $q\restrict[\gamma_{\alpha},\theta)$.

By the repeated application of Sublemma \ref{sublemmapd}, and using the fact that $\forced_\R$ ``$\dot{\Q}$ is ${<}\gamma_{\alpha}$-closed,'' one may build a descending sequence of conditions $\langle (r_*,\dot{q}_\xi)\mid\xi\leq\zeta\rangle$ in $\R*\dot{\Q}$ such that for each $\xi\leq\zeta$, the set 
$$\bar{D}^h_\xi:=\{r\leq r_*\mid (r,\dot{q}_\xi)\in D_\xi\}$$ 
is dense below $r^*$ in $\R=\P_{[\gamma_0,\gamma_{\alpha})}$. Let $p:=(r_*,\dot{q}_\zeta)$.

Suppose $p\in G^*_{[\gamma_0,\theta)}$ is $V[G_{\gamma_0}]$-generic for $\P_{[\gamma_0,\theta)}$. Fix an $a\in V_\beta$ such that $j(h)(a)^{G_{\gamma_0}}$ is a dense subset of $\widetilde{\P}_{[\gamma_0,\theta)}$. Since $j(h)(a)^{G_{\gamma_0}}$ must appear on the enumeration of dense sets we fixed above, there is a $\xi<\zeta$ such that $D^h_\xi=j(h)(a)^{G_{\gamma_0}}$. Since $\bar{D}_\xi$ is dense below $r^*$ in $\R=\P_{[\gamma_0,\gamma_{\alpha})}$ there is a condition $r\in G^*_{[\gamma_0,\gamma_{\alpha})}\cap\bar{D}_\xi$. By definition of $\bar{D}_\xi$, it follows that $(r,\dot{q}_\xi)\in D_\xi^{h}$. By padding $r$ with $\1$'s, one sees that there is an $\R$-name $\dot{b}$ such that $(r,\dot{b})\in G^*_{[\gamma_0,\theta)}$. Since $p=(r_*,\dot{q}_{\zeta})$ and $(r,\dot{b})$ are both in $G^*_{[\gamma_0,\theta)}$ they have a common extension $(r',\dot{q}')\in G^*_{[\gamma_0,\theta)}$. Since $(r',\dot{q}')\leq (r,\dot{q}_{\zeta})$, and since $r_*\forces \dot{q}_\zeta\leq\dot{q}_\xi$, it follows that $(r',\dot{q}')\leq (r,\dot{q}_\xi)$. Since $G^*_{[\gamma_0,\theta)}$ is a filter, one concludes that $(r,\dot{q}_\xi)\in G^*_{[\gamma_0,\theta)}\cap D^{h}_\xi$.
\end{proof}

Continuing with the proof of Lemma \ref{lemmapinfty}, I will now use Sublemma \ref{sublemmajoel} to construct the condition $p_\infty\in\P_{[\gamma_0,\theta)}$. Let $\langle f_\xi\mid\xi<\kappa^+\rangle\in V$ be a sequence of functions with domain $V_\kappa$ such that every dense subset of $\widetilde{\P}_{[\gamma_0,\theta)}$ in $M[G_{\gamma_0}]$ has a name of the form $j(f_\xi)(a)$ for some $\xi<\kappa^+$ and some $a\in V_\theta$. Let ${w}:\kappa^+\to\kappa^+\times\kappa^+$ be a bijection. It follows that ${w}\in M[G_{\gamma_0}]$ since ${w}\in V_\theta$. For each $\alpha<\kappa^+$ let ${w}(\alpha)=({w}(\alpha)_0,{w}(\alpha)_1)$. The function ${w}$ provides a well-ordering of pairs of the form $(f_\xi,\gamma_\alpha)$. Notice that the well-ordering is not in $M[G_{\gamma_0}]$ since the sequence $\langle f_\xi\mid\xi<\kappa^+\rangle$ is not in $M[G_{\gamma_0}]$. I will use this well-ordering of all pairs of the form $(f_\xi,\gamma_\alpha)$ of order type $\kappa^+$ to build a descending sequence of conditions $\langle p_\beta\mid\beta<\kappa^+\rangle$ in $V[G_{\gamma_0}]$ with $p_\beta\in \widetilde{\P}_{[\gamma_0,\theta)}$ such that if $p_\beta\in G^*_{[\gamma_0,\theta)}$ is $V[G_{\gamma_0}]$-generic for $\P_{[\gamma_0,\theta)}$, then $G^*_{[\gamma_0,\theta)}$ meets $D^{f_\xi}_a=j(f_\xi)(a)_{G_{\gamma_0}}$ for each $a\in V_{\gamma_\alpha}$ where ${w}(\beta)=(\xi,\alpha)$. Since the above mentioned well-ordering will not be in $M[G_{\gamma_0}]$, I will need the next lemma to build the descending sequence.

\begin{lemma}
The model $M[G_{\gamma_0}]$ is closed under $\kappa$-sequences in $V[G_{\gamma_0}]$.
\end{lemma}

\begin{proof}
Since $\P_\kappa$ is $\kappa$-c.c. in $V$ (by \cite[Theorem 16.30]{Jech:Book}), it follows that $M[G_\kappa]^\kappa\subseteq M[G_\kappa]$ in $V[G_\kappa]$. By Lemma \ref{lemmaclosuresacks} it follows that $M[G_\kappa][H_\kappa]^\kappa\subseteq M[G_\kappa][H_\kappa]$ in $V[G_\kappa][H_\kappa]$. Since the remaining forcing $\Q_{[\kappa^+,\bar{\kappa})}*\P_{[\bar{\kappa},\gamma_0)}$ is ${\leq}\kappa$-distributive in $V[G_\kappa][H_\kappa]$ (by Lemma \ref{lemmaeastonforsacks}) it follows that $M[G_{\gamma_0}]^\kappa\subseteq M[G_{\gamma_0}]$ in $V[G_{\gamma_0}]$.
\end{proof}

I will now use the bijection $w:\kappa^+\to\kappa^+\times\kappa^+$ defined above to build the descending sequence. Let $p_0$ be the condition obtained by applying Sublemma \ref{sublemmajoel} below the trivial condition to the function $h=f_{\xi}$ where $\xi={w}(0)_0$ and to the ordinal $\beta=\gamma_\alpha$ where $\alpha={w}(0)_1$. For successor stages, assume that $\langle p_\eta\mid\eta\leq\zeta\rangle$ has been constructed, where $\zeta<\kappa^+$. Let $p_{\zeta+1}\in\widetilde{\P}_{[\gamma_0,\theta)}$ be obtained by applying Sublemma \ref{sublemmajoel} below $p_\zeta$ to the function $h=f_\xi$ where $\xi={w}(\zeta+1)_0$ and to the ordinal $\beta=\gamma_\alpha$ where $\alpha={w}(\zeta+1)_1$. At limit stages $\zeta<\kappa^+$, assume $\langle p_\eta\mid\eta<\zeta\rangle$ has been constructed. The fact that $M[G_{\gamma_0}]^\kappa\subseteq M[G_{\gamma_0}]$ implies that the sequence $\langle p_\eta\mid\eta<\zeta\rangle$ is in $M[G_{\gamma_0}]$ since it has been constructed from an initial segment of $\langle f_\xi\mid\xi<\kappa^+\rangle$ and from $\langle \gamma_\alpha\mid\alpha<\kappa^+\rangle\in M[G_{\gamma_0}]$. Since $\widetilde{\P}_{[\gamma_0,\theta)}$ is ${<}\gamma_0$-closed in $M[G_{\gamma_0}]$, one may let $p_\zeta'\in \widetilde{\P}_{[\gamma_0,\theta)}$ be a lower bound of $\langle p_\beta\mid\beta<\zeta\rangle$. Now let $p_\zeta$ be obtained by applying Sublemma \ref{sublemmajoel} below $p_\zeta'$ to the function $f_\xi$ where $\xi={w}(\zeta)_0$ and the ordinal $\beta=\gamma_\alpha$ where $\alpha={w}(\zeta)_1$.

This defines the sequence $\langle p_\eta\mid\eta<\kappa^+\rangle$ in $V[G_{\gamma_0}]$ where $p_\eta\in\widetilde{\P}_{[\gamma_0,\theta)}\subseteq\P_{[\gamma_0,\theta)}$ for each $\eta<\kappa^+$. Let $p_\infty\in\P_{[\gamma_0,\theta)}$ be a lower bound of $\langle p_\eta\mid\eta<\kappa^+\rangle$.

Suppose $p_\infty\in G^*_{[\gamma_0,\theta)}$ is $V[G_{\gamma_0}]$-generic for $\P_{[\gamma_0,\theta)}$. Suppose $D\in M[G_{\gamma_0}]$ is a dense subset of $\widetilde{\P}_{[\gamma_0,\theta)}$. Then $D=D^{f_\xi}_a=j(f_\xi)(a)^{G_{\gamma_0}}$ for some $\xi<\kappa^+$ and where $a\in V_{\gamma_\alpha}$ for some $\alpha<\kappa^+$. Let $\zeta<\kappa^+$ with ${w}(\zeta)=({w}(\zeta)_0,{w}(\zeta)_1)=(\xi,\alpha)$. Since $p_\infty\leq p_\zeta$, it follows that $p_\zeta\in G^*_{[\gamma_0,\theta)}$ and hence, $G^*_{[\gamma_0,\theta)}$ meets $D^{f_\xi}_a$, by Sublemma \ref{sublemmajoel}.

This concludes the proof of Lemma \ref{lemmapinfty}. \hfill $\Box$

\bigskip

I will now show that there is an automorphic image of $G_{[\gamma_0,\theta)}$ containing $p_\infty$.

\begin{lemma}\label{lemmaaut}
Suppose $c\in \P_{[\gamma_0,\theta)}$. There is an automorphism $\pi:\P_{[\gamma_0,\theta)}\to\P_{[\gamma_0,\theta)}$ in $V[G_{\gamma_0}]$ such that $c\in\pi"G_{[\gamma_0,\theta)}$. 
\end{lemma}

\begin{proof}

Working in $V[G_{\gamma_0}]$, I claim each stage in the iteration $\P_{[\gamma_0,\theta)}$ is forced to be homogeneous over the previous stages. Let me argue that the Easton support product
$$\Q_\eta:=\Sacks(\eta,F(\eta))\times \prod_{\gamma\in(\eta,\bar{\eta})\cap\REG} \Add(\gamma,F(\gamma))$$
as defined in $V[G_\eta']$ is almost homogeneous in $V[G_\eta']$ where $G_\eta'$ is any generic for $\P_\eta$. It will suffice to argue that each factor in the product $\Q_\eta$ is almost homogeneous since automorphisms of each coordinate can be combined to give an automorphism of the product. Clearly, each factor of Cohen forcing $\Add(\gamma,F(\gamma))$ is almost homogeneouss. If $p,q\in\Sacks(\eta,F(\eta))$ let $f$ be an automorphism of $\Sacks(\eta,F(\eta))$ such that the support of $f(p)$ is disjoint from the support of $q$. Then $f(p)$ is compatible with $q$.
 
By Lemma \ref{lemmahomogeneousiteration}, to show that $\P_{[\gamma_0,\theta)}$ is almost homogeneous in $V[G_{\gamma_0}]$, it will suffice to show that at each stage $\eta\in[\gamma_0,\theta)$, the name $\dot{\Q}_\eta$ is a symmetric $\P_\eta$-name for the stage $\eta$ forcing. Let me fix an automorphism $f$ of $\P_{[\gamma_0,\eta)}$ and argue that $\forced_{\P_{[\gamma_0,\eta)}}f(\dot{\Q}_\eta)=\dot{\Q}_\eta$. There is a first order formula $\varphi(x_0,\ldots,x_n)$ such that $\forces_{\P_{[\gamma_0,\eta)}}$ ``$\forall x$ [$x\in\dot{\Q}_\eta$ if and only if $\varphi(x,\check{a}_1,\ldots,\check{a}_n)$]'' where $a_1,\ldots,a_n$ are elements of the ground model $V[G_{\eta_0}]$. Applying $f$ to the previous statement one obtains $\forces_{\P_{[\gamma_0,\eta)}}$ ``$\forall x$ [$x\in f(\dot{\Q}_\eta)$ if and only if $\varphi(x,\check{a}_1,\ldots,\check{a}_n)$].'' Thus, if $\dot{x}$ is a $\P_{[\gamma_0,\eta)}$-name in $V[G_{\gamma_0}]$, it follows that
$$\forces_{\P_\eta} \dot{x}\in\dot{\Q}_\eta \longleftrightarrow \varphi(\dot{x},\check{a}_1,\ldots,\check{a}_n)\longleftrightarrow \dot{x}\in f(\dot{\Q}_\eta)$$
and hence $\forces_{\P_\eta}\dot{\Q}_\eta= f(\dot{\Q}_\eta)$. Applying Lemma \ref{lemmahomogeneousiteration}, one concludes that $\P_{[\gamma_0,\theta)}$ is almost homogeneous in $V[G_{\gamma_0}]$. 

Now it follows by an easy density argument that every condition $p\in\P_{[\gamma_0,\theta)}$ can be extended to a condition $q\leq p$ such that there is an $f\in\Aut(\P_{[\gamma_0,\theta)})$ with $f(q)\leq c$. Therefore, by the genericity of $G_{[\gamma_0,\theta)}$, there is such a $q\in G_{[\gamma_0,\theta)}$ with such an $f\in\Aut(\P_{[\gamma_0,\theta)})$. Let $\pi:=f$. Since $\pi"G_{[\gamma_0,\theta)}$ is a filter and $\pi(q)\leq c$, it follows that $c\in \pi"G_{[\gamma_0,\theta)}$.
\end{proof}

\bigskip

As discussed above, one may use Lemmas \ref{lemmapinfty} and \ref{lemmaaut} to obtain $\widetilde{G}_{[\gamma_0,\theta)}\in V[G_{\gamma_0}][G_{[\gamma_0,\theta)}]$, an $M[G_{[\gamma_0,\theta)}]$-generic for $\widetilde{\P}_{[\gamma_0,\theta)}$. 

To finish lifting $j$ through $j(\P_\kappa)\cong \P_{\gamma_0}*\dot{\widetilde{\P}}_{[\gamma_0,\theta)}*\dot{\widetilde{\P}}_{[\theta,j(\kappa))}$, I will build an $M[G_{\gamma_0}][\widetilde{G}_{[\gamma_0,\theta)}]$-generic for $\widetilde{\P}_{[\theta,j(\kappa))}$ in $V[G_{\gamma_0}][G_{[\gamma_0,\theta)}]$. The following lemma will be required.

\begin{lemma}
$M[G_{\gamma_0}][\widetilde{G}_{[\gamma_0,\theta)}]$ is closed under $\kappa$-sequences in $V[G_{\gamma_0}][G_{[\gamma_0,\theta)}]$.
\end{lemma}
\begin{proof}
Since $\P_\kappa$ is $\kappa$-c.c., it follows by Lemma \ref{lemmachain} that $M[G_\kappa]$ is closed under $\kappa$-sequences in $V[G_\kappa]$. It is shown in \cite[Lemma 3.14]{FriedmanHonzik:EastonsTheoremAndLargeCardinals} and \cite[Lemma 3]{FriedmanThompson:PerfectTreesAndElementaryEmbeddings}, using a fusion argument, that $M[G_\kappa][H_\kappa]$ is closed under $\kappa$-sequences in $V[G_\kappa][H_\kappa]$. It will suffice to show that $M[G_{\gamma_0}][\widetilde{G}_{[\gamma_0,\theta)}]$ has every $\kappa$-sequence of ordinals in $V[G_{\gamma_0}][G_{[\gamma_0,\theta)}]$. Suppose $\vec{x}$ is a $\kappa$-sequence of ordinals in $V[G_{\gamma_0}][G_{[\gamma_0,\theta)}]$. Then since $\Q_{[\kappa^+,\bar{\kappa})}*\P_{[\bar{\kappa},\theta)}$ is ${\leq}\kappa$-distributive in $V[G_\kappa][H_\kappa]$, it follows that $\vec{x}\in V[G_\kappa][H_\kappa]$. Thus $\vec{x}\in M[G_\kappa][H_\kappa]\subseteq M[G_{\gamma_0}][\widetilde{G}_{[\gamma_0,\theta)}]$.
\end{proof}

Suppose $D$ is a dense subset of $\widetilde{\P}_{[\theta,j(\kappa))}$ in $M[G_{\gamma_0}][\widetilde{G}_{[\gamma_0,\theta)}]$. Let $\dot{D}\in M$ be a nice $\P_\theta$-name for $D$. Let $h$ be a function in $V$ with $\dom(h)=V_\kappa$ and $s\in V_\theta$ with $\dot{D}=j(h)(s)$. Without loss of generality, assume that $\ran(h)$ is contained in the set of nice names for dense subsets of a particular tail of $\P$. Since $\theta$ is singular, $\widetilde{\P}_{[\theta,j(\kappa))}$ is $\leq\theta$-closed in $M[G_{\gamma_0}][\widetilde{G}_{[\gamma_0,\theta)}]$. The collection $\mathcal{D}:=\{ j(h)(s)_{G_{\gamma_0}*\widetilde{G}_{[\gamma_0,\theta)}}\mid s\in V_\theta\}$ is in $M[G_{\gamma_0}][\widetilde{G}_{[\gamma_0,\theta)}]$. Since $\theta$ is a $\beth$-fixed point, there are at most $\theta$ dense subsets of $\widetilde{\P}_{[\theta,j(\kappa))}$ in $\mathcal{D}$. Thus, there is a single condition in $\widetilde{\P}_{[\theta,j(\kappa))}$ that meets every dense set in $\mathcal{D}$. Since there are at most $\kappa^+$ functions from $V_\kappa$ to nice names for dense subsets of a tail of $\P_\kappa$, and since every dense subset of $\widetilde{\P}_{[\theta,j(\kappa))}$ has a name in $M$ which is represented by such a function, the above procedure can be iterated to obtain a descending $\kappa^{+}$-sequence of conditions in $\widetilde{\P}_{[\theta,j(\kappa))}$ meeting every dense subset of $\widetilde{\P}_{[\theta,j(\kappa))}$ in $M[G_{\gamma_0}][\widetilde{G}_{[\gamma_0,\theta)}]$. Let $\widetilde{G}_{tail}$ be the $M[G_{\gamma_0}][\widetilde{G}_{[\gamma_0,\theta)}]$-generic filter for $\P_{tail}$ generated by this sequence.

Now let $j(G_\kappa):=G_{\gamma_0}*\widetilde{G}_{[\gamma_0,\theta)}*\widetilde{G}_{tail}$ and note that $j"G_\kappa\subseteq j(G_\kappa)$ since conditions in $G_\kappa$ have support bounded below the critical point of $j$. Hence by Lemma \ref{lemmaliftingcriterion}, the embedding lifts to 
$$j:V[G_\kappa]\to M[j(G_\kappa)]$$
in $V[G_{\gamma_0}][G_{[\gamma_0,\theta)}]$.

\subsection{Lifting $j$ Through $\Sacks(\kappa,F(\kappa))$.}

It remains to show that the embedding lifts further through the forcing $\P_{[\kappa,\lambda)}$. I will now argue that $j$ lifts through $\R_\kappa=\Sacks(\kappa,F(\kappa))^{V[G_\kappa]}$, the first factor of the stage $\kappa$ forcing. I will use the tuning fork method of \cite{FriedmanThompson:PerfectTreesAndElementaryEmbeddings} to construct an $M[j(G_\kappa)]$-generic for $j(\R_\kappa)=\Sacks(j(\kappa),j(F(\kappa)))^{M[j(G_\kappa)]}$ in $V[G_\kappa][H_\kappa]$ that satisfies the lifting criterion in Lemma \ref{lemmaliftingcriterion}. Say that $t\subseteq 2^{<j(\kappa)}$ is a \emph{tuning fork that splits at $\kappa$} if and only if $t=t^0\cup t^1$ where $t^0$ and $t^1$ are two distinct cofinal branches of $2^{<j(\kappa)}$ such that $t^0\cap\kappa=t^1\cap\kappa$, $t^0(\kappa)=0$, and $t^1(\kappa)=1$. For $\alpha<j(F(\kappa))$ let 
$$t_\alpha:=\bigcap\{j(p)(\alpha)\mid p\in H_\kappa\}.$$
The next lemma is key.
\begin{lemma}\label{lemmatuningfork}
If $\alpha\in j"F(\alpha)$ then $t_\alpha$ is a tuning fork that splits at $\kappa$. Otherwise, if $\alpha<j(F(\kappa))$ is not in the range of $j$, then $t_\alpha$ is a cofinal branch through $2^{<j(\kappa)}$. 
\end{lemma}

\begin{proof}
The following proof follows \cite{FriedmanThompson:PerfectTreesAndElementaryEmbeddings} closely, except that here Lemma \ref{lemmawoodinmenas} is required. Working in $V[G_\kappa]$, let 
$$X:=\bigcap \{j(C)\mid \textrm{$C\subseteq\kappa$ is club and $C\in V$}\}.$$
First let me show that $X=\{\kappa\}$. If $\alpha<\kappa$ then clearly $\alpha\notin X$ since there is a closed unbounded subset $C$ of $\kappa$ whose least element is greater than $\alpha$, and thus $\alpha\notin j(C)$. Since the limit cardinals below $\kappa$ form a closed unbounded subset of $\kappa$ it follows that any element of $X$ must be a limit cardinal in $M[j(G_\kappa)]$ which is greater than or equal to $\kappa$. Suppose $\lambda<j(\kappa)$ is a limit cardinal and $\lambda>\theta$. Then $\lambda=j(h)(a)$ for some function $h:V_\kappa\to \kappa$ in $V[G_\kappa]$ and some $a\in V_\theta$. Let $C_h:=\{\gamma<\kappa\mid\textrm{$\gamma$ is a limit cardinal and $h"V_\gamma\subseteq\gamma$}\}$. Then $C_h$ is a closed unbounded subset of $\kappa$ and $\lambda\notin j(C_h)$ since $\lambda>\theta$ and $j(h)"V_\lambda\not\subseteq\lambda$. Now suppose $\kappa<\lambda\leq\theta$. Above, the function $\ell$ is chosen using Lemma \ref{lemmawoodinmenas} so that $\ell:\kappa\to\kappa$ and $j(\ell)(\kappa)=\theta$. Then $C_\ell:=\{\gamma<\kappa\mid\textrm{$\ell"\gamma\subseteq\gamma$}\}$ is a closed unbounded subset of $\kappa$ in $V[G_\kappa]$ and $\lambda\notin j(C_\ell)$ since $\theta\in j(\ell)"\lambda$ and this implies $j(\ell)"\lambda\not\subseteq\lambda$. This shows that $X\subseteq\{\kappa\}$. Clearly $\kappa\in X$ since for each closed unbounded $C\subseteq\kappa$ in $V[G_\kappa]$, $j(C)\cap\kappa=C$.

The rest of the proof is exactly as in \cite{FriedmanThompson:PerfectTreesAndElementaryEmbeddings} and \cite{FriedmanHonzik:EastonsTheoremAndLargeCardinals}.

Let $C$ be any closed unbounded subset of $\kappa$ in $V[G_\kappa]$. Choose $\alpha<j(F(\kappa))$ and write $\alpha=j(f)(a)$ where $f:V_\kappa\to 	F(\kappa)$ and $a\in V_\theta$. It is easy to show that the following set is dense in $\Sacks(\kappa,F(\kappa))$.
$$D_C=\{p\in\Sacks(\kappa,F(\kappa))\mid\xi\in\ran(f)\implies C(p(\xi))\subseteq C\}$$
Thus there is a $p\in H_\kappa\cap D_C$ with $C(j(p)(\alpha))\subseteq j(C)$. Since $C$ was an arbitrary closed unbounded subset of $\kappa$, this, together with the fact that $X=\{\kappa\}$, implies that $t_\alpha$ can only possibly split at $\kappa$. If $\alpha\in\ran(j)$ then since $\kappa$ is a limit point of $j(C)$ for every closed unbounded $C\subseteq\kappa$ in $V[G_\kappa]$, it follows that $t_\alpha$ splits at $\kappa$ and is a tuning fork. 

If $\alpha\notin\ran(j)$ then $\ran(f)$ must have size $\kappa$ since otherwise $\alpha\in j(\ran(f))= j"\ran(f)$. Let $\langle \bar{\alpha}_i\mid i<\kappa\rangle$ enumerate $\ran(f)$. Then $j(\langle \bar{\alpha}_i\mid i<\kappa\rangle)=\langle \alpha_i\mid i<j(\kappa)\rangle$ in an enumeration of $\ran(j(f))$. It is easy to see that the set of conditions $p\in\Sacks(\kappa,F(\kappa))$ such that for each $i<\kappa$, the least splitting level of $p(\bar{\alpha}_i)$ is above level $i$ is dense. Thus there is a $p\in H_\kappa$ such that for each $i<j(\kappa)$ the least splitting level of $j(p)(\alpha_i)$ is beyond level $i$. Since $\alpha\notin\ran(j)$ it follows that $\alpha=\alpha_i$ for some $i\in[\kappa,j(\kappa))$. It follows that the first splitting level of $j(p)(\alpha)$ is above $\kappa$. Thus, $t_\alpha$ is a cofinal branch.
\end{proof}

Each $t_\alpha$ generates an $M[j(G_\kappa)]$-generic filter for $j(\R_\kappa)$ as follows. For $\alpha\in j"F(\kappa)$, let $t_\alpha^0$ and $t_\alpha^1$ be the left-most and right-most branches of $t_\alpha$ respectively; that is, for $k\in\{0,1\}$ let
$$t^k_\alpha:=\{s\in t_\alpha\mid \kappa\in\dom(s)\implies s(\kappa)=k\}.$$
For $\alpha<j(F(\kappa))$ not in the range of $j$, let $t^0_\alpha:=t_\alpha$ be the cofinal branch in Lemma \ref{lemmatuningfork}.
Let 
$$g:=\{\widetilde{p}\in j(\R_\kappa)\mid \forall \alpha<j(F(\kappa))\ t^0_\alpha\subseteq \widetilde{p}(\alpha)\}.$$
It is easy to check that $j"H_\kappa\subseteq g$, so to show that $j$ lifts through $\R_\kappa$ it remains to show that $g$ is $M[j(G_\kappa)]$-generic for $j(\R_\kappa)$. For this the following two definitions will be used, both of which are given in \cite{FriedmanThompson:PerfectTreesAndElementaryEmbeddings}. Suppose $ p\in \Sacks(\kappa,F(\kappa))^{V[G_\kappa]}$, $S\subseteq F(\kappa)$ with $|S|^{V[G_\kappa]}<\kappa$. Friedman and Thompson say that an \emph{$(S,\alpha)$-thinning of $ p$} is an extension of $ p$ obtained by thinning each $ p(\xi)$ for $\xi\in S$ to the subtree 
$$ p(\xi)\restrict s_\xi:=\{s\in p(\xi)\mid s_\xi\subseteq s\textrm{ or } s\subseteq s_\xi\}$$
where $s_\xi$ is some particular node of $p(\xi)$ on the $\alpha$-th splitting level of $p(\xi)$. A condition $p\in \Sacks(\kappa, F(\kappa))^{V[G_\kappa]}$ is said to \emph{reduce} a dense subset $D$ of $\Sacks(\kappa, F(\kappa))^{V[G_\kappa]}$ if and only if for some $S\subseteq  F(\kappa)$ of size less than $\kappa$ in $V[G_\kappa]$, any $(S,\alpha)$-thinning of $p$ meets $D$.

Let me now argue that $g$ is $M[j(G_\kappa)]$-generic for the poset $j(\R_\kappa)=\Sacks(j(\kappa),j( F(\kappa)))^{M[j(G_\kappa)]}$. Suppose $D$ is a dense subset of $j(\R_\kappa)$ in the model $M[j(G_\kappa)]$. Then by Lemma \ref{lemmaseedpreservation} one can write $D=j(h)(a)$ where $h\in V[G_\kappa]$ is a function from $V_\kappa$ to the collection of dense subsets of $\Sacks(\kappa,$ $F(\kappa))^{V[G_\kappa]}$ and $a\in V_\theta$. Let $\langle D_\beta\mid\beta<\kappa\rangle\in V[G_\kappa]$ enumerate the range of $h$. One may show, as in \cite{FriedmanThompson:PerfectTreesAndElementaryEmbeddings} that any condition $p\in \Sacks(\kappa, F(\kappa))^{V[G_\kappa]}$ can be extended to $q\leq p$ which reduces each $D_\beta$ for $\beta<\kappa$. This implies that the following is a dense subset of $\Sacks(\kappa, F(\kappa))^{V[G_\kappa]}$.
$$D':=\{p\in\R_\kappa\mid \textrm{$p$ reduces each $D_\beta$ for $\beta<\kappa$}\}$$
Thus one may choose a condition $p\in H\cap D'$. By elementarity $j(p)$ reduces each dense subset of $j(\R_\kappa)$ in the range of $j(h)$; in particular, $j(p)$ reduces $D=j(h)(a)$. Thus it follows that there is an $S\subseteq j( F(\kappa))$ of size less than $j(\kappa)$ and an $\alpha<j(\kappa)$ such that any $(S,\alpha)$-thinning of $j(p)$ meets $D$. For each $\xi\in S$ let $\widetilde{q}(\xi)$ be the thinning of $j(p)(\xi)$ obtained by choosing an initial segment of $t_\xi^0$ on the $\alpha$-th splitting level of $j(p)(\xi)$. For $\xi\in j( F(\kappa))\setminus S$ let $\widetilde{q}(\xi):=j(p)(\xi)$. The fact that $\widetilde{q}$ is a condition in $j(\R_\kappa)$ will follow from the next lemma, which appears in \cite{FriedmanThompson:PerfectTreesAndElementaryEmbeddings}.

\begin{lemma}\label{lemmathinned}
For any $\beta<j(\kappa)$ and any subset $S$ of $j( F(\kappa))$ of size at most $j(\kappa)$ in $M[j(G_\kappa)]$, the sequence $\langle t_\xi^0\restrict\beta\mid\xi\in S\rangle$ belongs to $M[j(G_\kappa)]$. 
\end{lemma}

\begin{proof}

Write $\beta=j(f_0)(a)$ where $f_0:V_\kappa\to \kappa$ and $a\in V_\theta$. Let $C=\{\lambda<\kappa\mid\textrm{$f_0"V_\lambda\subseteq\lambda$ and $\lambda$ is a limit cardinal}\}$. By Lemma \ref{lemmaseedpreservation} it follows that $S=j(f)(b)$ where $f:V_\kappa\to[F(\kappa)]^{\leq\kappa}$ and $b\in V_\theta$. Since $S\subseteq j(\bigcup\ran(f))$ it can be assumed without loss of generality that $S=j(\bar{S})$ for some $\bar{S}\in[F(\kappa)]^{\leq\kappa}$. Let $\langle \bar{\alpha}_i\mid i<\kappa\rangle$ be an enumeration of $\bar{S}$. Then $j(\langle \bar{\alpha}_i\mid i<\kappa\rangle)=\langle\alpha_i\mid i<j(\kappa)\rangle$ is an enumeration of $S$. One can easily see that 
$$D=\{\bar{p}\in\Sacks(\kappa, F(\kappa))\mid\textrm{for each $i<\kappa$, $C(\bar{p}(\bar{\alpha}_i))\subseteq C\setminus(i+1)$}\}$$
is a dense subset of $\Sacks(\kappa, F(\kappa))$. Let $\bar{p}\in H_\kappa\cap D$. Then for each $i<j(\kappa)$, $C(j(\bar{p})(\alpha_i))\subseteq C\setminus (i+1)$. Thus, for each $\alpha_i$, the tree $j(\bar{p})(\alpha_i)$ has no splits between $\kappa$ and $\alpha$. If $\kappa\leq i< j(\kappa)$ then $j(\bar{p})(\alpha_i)$ does not split between $0$ and $\alpha$. If $\kappa\leq i <j(\kappa)$ then $t^0_{\alpha_i}\restrict \alpha$ is the unique element of $j(\bar{p})(\alpha_i)$ of length $\alpha$. If $i<\kappa$, then $t^0_{\alpha_i}\restrict\alpha$ is the unique element of $j(\bar{p})(\alpha_i)$ that extends $t^0_{\alpha_i}\restrict\kappa$ and takes on value $0$ at $\kappa$. 
\end{proof}

By Lemma \ref{lemmathinned}, $\widetilde{p}$ is in $M[j(G_\kappa)]$ and is thus a condition in $j(\R_\kappa)$. Furthermore, $\widetilde{p}$ meets $D$ and since $t^0_\xi\subseteq\widetilde{p}(\xi)$ for each $\xi<F(\kappa)$, it follows that $\widetilde{p}$ is in $g$. This establishes that $g$ is $M[j(G_\kappa)]$-generic for $j(\R_\kappa)$. Thus the embedding lifts to $j:V[G_\kappa][H_\kappa]\to M[j(G_\kappa)][j(H_\kappa)]$.

\subsection{Lifting $j$ Through $\Q_{[\kappa^+,\bar{\kappa})}*\P_{[\bar{\kappa},\delta)}$.}

By Lemma \ref{lemmaeastonforsacks}, the poset $\Q_{[\kappa^+,\bar{\kappa})}*\P_{[\bar{\kappa},\delta)}$ is $\leq\kappa$-distributive in $V[G_\kappa][H_\kappa]$. Thus, from Lemma \ref{lemmalambdadist} one sees that $j"H_{[\kappa^+,\bar{\kappa})}*G_{[\bar{\kappa},\delta)}$ generates an $M[j(G_\kappa)]$ $[j(H_\kappa)]$-generic filter for $j(\Q_{[\kappa^+,\bar{\kappa})}*\P_{[\bar{\kappa},\delta)})$, call it $j(H_{[\kappa^+,\bar{\kappa})}*G_{[\bar{\kappa},\delta)})$. Thus $j$ lifts to $j:V[G_\delta]\to M[j(G_\delta)]$ where $j(G_\delta):=j(G_\kappa)*(j(H_\kappa)\times j(H_{[\kappa^{+},\bar{\kappa})}))*j(G_{[\bar{\kappa},\delta)})$.

\subsection{Verifying strongness for $A$}

Let me argue that the lifted embedding $j:V[G_\delta]\to M[j(G_\delta)]$ satisfies $j(A)\cap\mu = A\cap \mu$. This will follow from the next fact. 
\begin{fact}\label{factequation}\ 
\begin{itemize}
\item[$(1)$] $j(\dot{A})\cap V_\theta = \dot{A}\cap V_\theta$
\item[$(2)$] $j(G_\delta)=G_{\gamma_0}*\widetilde{G}_{[\gamma_0,\theta)}*\widetilde{G}_{[\theta,j(\kappa))}$ agrees with $G_\delta$ up to $\mu'$ since $\mu'<\gamma_0$.
\item [$(3)$] $j(u)\restrict \mu' = u\restrict\mu'$
\end{itemize}
\end{fact}
Using the above fact, one has the following.
\begin{align*}
A\cap\mu & = \dot{A}^{G_\delta}\cap\mu \\
 & = (\dot{A}\cap V_{\mu'})^{G_{\mu'}}\cap\mu \tag{\textrm{using the definition of $u$}}\\
 & = (j(\dot{A})\cap V_{\mu'})^{G_{\mu'}}\cap\mu \tag{\textrm{by Fact \ref{factequation}(1)}}\\
 & = j(\dot{A})^{j(G_\delta)}\cap\mu \tag{\textrm{by Fact \ref{factequation}(2) and (3)}}\\
 & = j(A)\cap \mu \\
\end{align*}
This completes the proof of Theorem \ref{theoremwoodin}.
\end{proof}

\section{Conclusion and some open questions}

Cummings and Shelah \cite{CummingsShelah:CardInvariantsAboveContinuum} generalized Easton's Theorem \cite{Easton:PowersOfRegularCardinals}, by forcing to control not only the continuum function $\kappa\mapsto 2^\kappa$ on the regular cardinals, but to control the bounding number $\mathfrak{b}(\kappa)$ and dominating number $\mathfrak{d}(\kappa)$ on regular cardinals. Thus a natural extension of Question \ref{question} is: To what extent can one control the function $\kappa\mapsto (\mathfrak{b}(\kappa),\mathfrak{d}(\kappa),2^\kappa)$ on the regular cardinals by forcing while preserving large cardinals? In particular, to what extent can one control the behavior of $\kappa\mapsto (\mathfrak{b}(\kappa),\mathfrak{d}(\kappa),2^\kappa)$ on the regular cardinals while preserving a Woodin cardinal?

Recall that $\delta$ is a \emph{Shelah cardinal} if for every function $f:\delta\to\delta$ there is a transitive class $N$ and an elementary embedding $j:V\to N$ with critical point $\delta$ such that $V_{j(f)(\delta)}\subseteq N$. Since every Shelah cardinal is measurable (and more), it follows that under the assumption that $\delta$ is a Shelah cardinal, the continuum function has less freedom than under the assumption that $\delta$ is a Woodin cardinal. If $\delta$ is a Shelah cardinal, which Easton functions can one force to agree with the continuum function while preserving the Shelahness of $\delta$?




\begin{thebibliography}{Cum10}

\bibitem[Cod12]{Cody:Dissertation}
Brent Cody.
\newblock {\em Some results on large cardinals and the continuum function}.
\newblock PhD thesis, The Graduate Center of the City University of New York,
  2012.

\bibitem[CS95]{CummingsShelah:CardInvariantsAboveContinuum}
James Cummings and Saharon Shelah.
\newblock Cardinal invariants above the continuum.
\newblock {\em Annals of Pure and Applied Logic}, 75(3):251--268, 1995.

\bibitem[Cum92]{Cummings:AModelInWhichGCH}
James Cummings.
\newblock A model in which {GCH} holds at successors but fails at limits.
\newblock {\em Transactions of the American Mathematical Society},
  392(1):1--39, 1992.

\bibitem[Cum10]{Cummings:Handbook}
James Cummings.
\newblock Iterated forcing and elementary embeddings.
\newblock In Akihiro Kanamori and Matthew Foreman, editors, {\em Handbook of
  Set Theory}, volume~2, chapter~14, pages 775---883. Springer, 2010.

\bibitem[DF08]{DobrinenFriedman:HomogeneousIteration}
Natasha Dobrinen and Sy~David Friedman.
\newblock Homogeneous iteration and measure one covering relative to {HOD}.
\newblock {\em Archive of Mathematical Logic}, 47(7--8):711--718, 2008.

\bibitem[Eas70]{Easton:PowersOfRegularCardinals}
William~B. Easton.
\newblock Powers of regular cardinals.
\newblock {\em Annals of Mathematical Logic}, 1:139--178, 1970.

\bibitem[FH08]{FriedmanHonzik:EastonsTheoremAndLargeCardinals}
Sy~David Friedman and Radek Honzik.
\newblock Easton's theorem and large cardinals.
\newblock {\em Annals of Pure and Applied Logic}, 154(3):191--208, 2008.

\bibitem[FT08]{FriedmanThompson:PerfectTreesAndElementaryEmbeddings}
Sy~David Friedman and Katherine Thompson.
\newblock Perfect trees and elementary embeddings.
\newblock {\em Journal of Symbolic Logic}, 73(3):906--918, 2008.

\bibitem[Ham97]{Hamkins:CanonicalSeedsAndPrickryTrees}
Joel~David Hamkins.
\newblock Canonical seeds and prikry trees.
\newblock {\em The Journal of Symbolic Logic}, 62(2):373--396, 1997.

\bibitem[HW00]{HamkinsWoodin:SmallForcing}
Joel~David Hamkins and W.~Hugh Woodin.
\newblock Small forcing creates neither strong nor woodin cardinals.
\newblock {\em Proceedings of the American Mathematical Society},
  128(10):3025--3029, 2000.

\bibitem[Jec03]{Jech:Book}
Thomas Jech.
\newblock {\em Set Theory: The Third Millennium Edition, revised and expanded}.
\newblock Springer, 2003.

\bibitem[Kan80]{Kanamori:PerfectSetForcing}
Akihiro Kanamori.
\newblock Perfect-set forcing for uncountable cardinals.
\newblock {\em Annals of Mathematical Logic}, 19(1--2):97--114, 1980.

\bibitem[Kan03]{Kanamori:Book}
Akihiro Kanamori.
\newblock {\em The Higher Infinite: Large Cardinals in Set Theory from Their
  Beginnings}.
\newblock Springer, second edition, 2003.

\bibitem[Men76]{Menas:ConsistencyResultsConcerningSupercompactness}
Telis~K. Menas.
\newblock Consistency results concerning supercompactness.
\newblock {\em Transactions of the American Mathematical Society}, 223:61--91,
  1976.

\end{thebibliography}

\end{document}